\documentclass{amsart}
\usepackage{amsmath, amscd, amssymb, amsthm}
\usepackage{bbm}
\usepackage{latexsym}
\usepackage{amsfonts}
\usepackage{graphicx}
\usepackage[all,cmtip]{xy}

\usepackage[
dvipdfm, colorlinks=true, linkcolor=black,breaklinks=true,
urlcolor=blue, citecolor=black
]{hyperref}%
\usepackage{geometry}
\newtheorem{theorem}{Theorem}[section]
\newtheorem{lemma}[theorem]{Lemma}

\newtheorem{proposition}[theorem]{Proposition}

\newtheorem{question}[theorem]{Question}

\newtheorem*{definition}{Definition}
\newtheorem{conjecture}[theorem]{Conjecture}

\newcommand{\be}{\begin{equation}}
\newcommand{\ee}{\end{equation}}
\newcommand{\bea}{\begin{eqnarray}}
\newcommand{\eea}{\end{eqnarray}}

\newcommand{\eps}{\varepsilon}

\newcommand{\vs}{\vspace{0.5cm}}

\newcommand{\vsv}{\vspace{0.12cm}}

\def\XXint#1#2#3{{\setbox0=\hbox{$#1{#2#3}{\int}$ }
\vcenter{\hbox{$#2#3$ }}\kern-.6\wd0}}

\begin{document}

\title{On compact Hermitian manifolds with\\ flat Gauduchon connections}

\author{Bo Yang}

\address{Bo Yang. School of Mathematical Sciences, Xiamen University, Xiamen, Fujian, 361005, China.}
\email{{boyang@xmu.edu.cn}}

\author{Fangyang Zheng}


\address{Fangyang Zheng. Department of Mathematics, The Ohio State
University, 231 West 18th Avenue, Columbus, OH 43210, USA}

\email{{zheng.31@osu.edu}}

\begin{abstract}
Given a Hermitian manifold $(M^n,g)$, the Gauduchon connections are
the one parameter family of Hermitian connections joining the Chern
connection and the Bismut connection. We will call $\nabla^s =
(1-\frac{s}{2})\nabla^c + \frac{s}{2}\nabla^b$ the $s$-Gauduchon
connection of $M$, where $\nabla^c$ and $\nabla^b$ are respectively
the Chern and Bismut connections. It is natural to ask when a
compact Hermitian manifold could admit a flat $s$-Gauduchon
connection. This is related to a question asked by Yau \cite{Yau}.
The cases with $s=0$ (a flat Chern connection) or $s=2$ (a flat
Bismut connection) are classified respectively  by Boothby
\cite{Boothby} in the 1950s or by Q. Wang and the authors recently
\cite{WYZ}. In this article, we observe that if either $s\geq
4+2\sqrt{3} \approx 7.46$ or $s\leq 4-2\sqrt{3}\approx 0.54$ and
$s\neq 0$, then $g$ is K\"ahler. We also show that, when $n=2$, $g$
is always K\"ahler unless $s=2$. Note that non-K\"ahler compact
Bismut flat surfaces are exactly those isosceles Hopf surfaces by
\cite{WYZ}.
\end{abstract}

\maketitle

\tableofcontents

\markleft{Gauduchon flat manifolds}

\markright{Gauduchon flat manifolds}

\section{Introduction}

S.T.\! Yau (\cite{Yau}) asked an interesting question on Hermitian
geometry.
\begin{question}[Problem 87 in \cite{Yau}] \label{Yau problem}
If the holonomy group of a compact Hermitian manifold can be reduced
to a proper subgroup of $U(n)$, can we say something nontrivial
about the manifold? The problem is that the connection need not to
be Riemannian.
\end{question}

Recall that on a Hermitian manifold $(M^n,g)$, there are a lot of
Hermitian connections, namely, linear connections $\nabla$
satisfying $\nabla g=0$ and $\nabla J=0$, where $J$ is the almost
complex structure of the complex manifold $M^n$.

\vsv

An important special case of Question \ref{Yau problem} is when the
Hermitian connection under consideration is flat, hence its holonomy
group is discrete. The question states:

\begin{question}\label{Yau problem more}
Classify compact Hermitian manifolds which admit flat Hermitian
connections.
\end{question}

One of the main difficulties in answering Question \ref{Yau problem
more} is that the linear space of Hermitian connections on a general
Hermitian manifold is of infinite dimension. Note that there are
three Hermitian connections relatively well studied in the
literature: the Chern connection $\nabla^c$, the Bismut connection
$\nabla^b$ (introduced in \cite{Bismut} and \cite{Yano}), and
$\nabla^{lv}$ which is the projection onto the holomorphic tangent
bundle of the Riemannian (Levi-Civita) connection of $g$. The latter
was also called the associated connection (\cite{GK}), Levi-Civita
connection (\cite{LY}, \cite{LY1}),  or the first canonical
connection, etc. It is well-known that these three connections form
a straight line in the space of all Hermitian connections, namely,
$\nabla^{lv}$ is the arithmetic average of the other two. In the
rest of this paper, we will fix the following notation which is
partly motivated by the work of Gauduchon \cite{Gauduchon2}:

\begin{definition} [Gauduchon connections]
The {\em $s$-Gauduchon connection} of $(M^n,g)$ is defined to be the
Hermitian connection
$ \nabla^s = (1-\frac{s}{2}) \nabla^c + \frac{s}{2} \nabla^b $,
where $s\in {\mathbb R}$.
\end{definition}

So $\nabla^0=\nabla^c$, $\nabla^2=\nabla^b$, and
$\nabla^1=\nabla^{lv}$. When $g$ is K\"ahler, every $s$-Gauduchon
connection $\nabla^s \equiv \nabla^c$ and is equal to the Riemannian
connection. When $g$ is not K\"ahler, $\nabla^s\neq \nabla^{s'}$
whenever $s\neq s'$. From the differential geometric point of view,
we want to study the curvature tensor $R^s$ of $\nabla^s$. As a
first step to understand Question \ref{Yau problem more}, we would
like to know what kind of compact Hermitian manifolds can have flat
$s$-Gauduchon connection $\nabla^s$.

\vsv

When $s=0$, the question is well understood and well-known. In 1958,
Boothby (\cite{Boothby}) proved that compact Hermitian manifolds
with flat Chern connection are exactly the quotients of complex Lie
groups, equipped with left invariant Hermitian metrics. An important
subset of this is the complex parallelizable manifolds, which are
classified by H.-C. Wang \cite{Wang}.

\vsv

For $s=2$, the recent work \cite{WYZ} of Q. Wang and the authors
classified all (including noncompact ones) Bismut flat manifolds.
Those compact Bismut flat manifolds are exactly the (finite
undercover of) compact local Samulson spaces. In more details, given
any compact Bismut flat manifold $M^n$, its universal cover is a
Samelson space, namely, $G\times {\mathbb R}^k$ equipped with a
bi-invariant metric and a left invariant complex structure. Here $G$
is a simply-connected compact semisimple Lie group, and $0\leq k
\leq 2n$. In particular, compact non-K\"ahler Bismut flat surfaces
are exactly those isosceles Hopf surfaces, and in dimension three,
their universal cover is either a central Calabi-Eckmann threefold
$S^3\times S^3$, or $({\mathbb C}^2\setminus \{ 0\} )\times {\mathbb
C}$. The readers are referred to \cite{WYZ} for more details.

\vsv

For $s=1$, Ganchev and Kassabov \cite{GK} proved an interesting
local characterization theorem for Hermitian manifolds with flat
$\nabla^1$, which they called the associated connection. More precisely,
they showed that for any $n\geq 2$, if a Hermitian manifold $(M^n,g)$
has flat $\nabla^1$ and is conformal to a K\"ahler metric
$\tilde{g} = e^{2u}g$, then $\tilde{g}$ has constant holomorphic
sectional curvature. Conversely, given any K\"ahler metric $\tilde{g}$
with constant holomorphic sectional curvature, there always exists $g$
conformal $\tilde{g}$ such that $g$ has flat $\nabla^1$.

\vsv

To sum up, we propose the following version of Question \ref{Yau
problem more}:

\begin{conjecture}\label{YZconj}
If $s\neq 0, 2$, then any compact Hermitian manifold $(M^n,g)$ which
admits a flat $s$-Gauduchon connection must be K\"ahler, thus being
a finite undercover
 of a flat complex torus.
\end{conjecture}

The main purpose of this article is to confirm the above conjecture
in the $n=2$ case, namely, we have the following

\begin{theorem}\label{intro1}
Let $(M^2,g)$ be a compact Hermitian surface with a flat
$s$-Gauduchon connection $\nabla^s$. If $s\neq 2$, then $g$ is
K\"ahler.
\end{theorem}

So $M^2$ is either a flat complex torus or a flat hyperelliptic
surface. Note that for $s=0$, the Chern flat case, $g$ is still
K\"ahler since $n=2$, while when $s=2$, $M^2$ can be non-K\"ahler.
Such surfaces are exactly those isosceles Hopf surfaces
(\cite{WYZ}). For $n\geq 3$, we are able to prove the following:

\begin{theorem}\label{intro2}
Let $(M^n,g)$ be a compact Hermitian manifold with a flat
$s$-Gauduchon connection $\nabla^s$. If either $s\geq  4+2\sqrt{3}$,
or $s\leq 4-2\sqrt{3}$ and $s\neq 0$, then $g$ is K\"ahler.
\end{theorem}

In other words, if $(M^n,g)$ is compact, $s$-Gauduchon flat with
$s\neq 0$, and is non-K\"ahler, then $s$ must lie in the interval
$(4-2\sqrt{3}, 4+2\sqrt{3})$. Note however that this interval contains
the interesting cases $s=2$ (Bismut), $s=1$, and $s=\frac{2}{3}$ (see below).

\vsv

If we take the dimension into account, then the two constants $4\pm
2\sqrt{3}$ can be slightly improved. For $n\geq 3$, let us denote by
$$ a_n^{\pm} = \frac{1}{n} \big[ 4(n-1) \pm 2 \sqrt{ 3n^2-7n+4} \big] ,
\ \ \ b_n^{\pm} = \frac{1}{(n+1)} \big[ 4(n-1) \pm 2 \sqrt{
3n^2-8n+5} \big]$$ Then we have $a_n^-<b_n^- <0.6$ and both
sequences $(a_n^-)$ and $(b_n^-)$ are monotonically decreasing and
approaching $4-2\sqrt{3}$ when $n\rightarrow \infty $. Similarly,
$3.4<b_n^+ < a_n^+$, and both sequences $(a_n^+)$ and $(b_n^+)$ are
monotonically increasing and approaching $4+2\sqrt{3}$ when
$n\rightarrow \infty $. The statement of Theorem \ref{intro2} can be
slightly improved by

\begin{theorem}\label{intro3}
Let $(M^n,g)$ be a compact Hermitian manifold with a flat
$s$-Gauduchon connection $\nabla^s$. If either $s>a_n^+ $, or $s<
a_n^-$ and $s\neq 0$, then $g$ is K\"ahler.
\end{theorem}

\vsv

When the metric $g$ is locally conformally K\"ahler, its torsion
tensor takes a simple form, so the above type consideration leads to
the conclusion that $g$ will be K\"ahler for all values of $s$ except
possibly when $s=b_n^+$ or $s=b_n^-$. That is, we have

\begin{theorem}\label{intro4}
Let $(M^n,g)$ be a compact Hermitian manifold with a flat
$s$-Gauduchon connection $\nabla^s$. Assume $n\geq 3$ and $g$ is
locally conformal K\"ahler. If $s\neq b_n^{\pm}$, then $g$ is
K\"ahler.
\end{theorem}

\vsv

The main idea of proving Theorem \ref{intro1} is:  when $n=2$, the
torsion $1$-form $\eta$ contains all the information about the
torsion tensor. So the above type of consideration in Theorem
\ref{intro4} can be pushed further to lead to the K\"ahlerness of
$g$ for all $s$ values (other than $0$ and $2$) except one value:
$s=\frac{2}{3}$. By a Bochner type argument, one can conclude the
K\"ahlerness in the case $s=\frac{2}{3}$ as well, thus proving
Theorem \ref{intro1}.

\vsv

Note that in all dimensions, the connection $\nabla^{\frac{2}{3}}$
distinguishes itself with the property that it has the smallest total
torsion amongst all Gauduchon connections. For that reason, we will
also call the $\frac{2}{3}$-Gauduchon connection $\nabla^{\frac{2}{3}}$
the {\em minimal Gauduchon connection.}

\vsv

Recall that a Hermitian metric $g$ is called {\em balanced,} if
$d \omega^{n-1} =0$, where $\omega$ is the K\"ahler form of $g$.
$g$ is said to be {\em Gauduchon} if
$\partial \overline{\partial} \omega^{n-1} =0$.  A local property
about Gauduchon flat manifolds worth mentioning is the following:

\begin{proposition}\label{intro5}
Let $(M^n,g)$ be a Hermitian manifold with a flat $s$-Gauduchon
connection. If $s=\frac{1}{2}$, then $g$ is K\"ahler. Also, if
$s\neq 0$ and $g$ is balanced, then $g$ is K\"ahler.
\end{proposition}

\vsv Let us remark that a noncompact version of Conjecture
\ref{YZconj} is much more subtle. If we focus on the Chern flat case
($s=0$), there are example of noncompact Hermitian surfaces
(incomplete ones in \cite{Boothby} and complete ones in \cite{WYZ})
without parallel Chern torsion, hence they do not come from
quotients of complex Lie groups with left invariant metrics. On the
other hand, noncompact (not necessarily complete) Bismut flat
($s=2$) Hermitian manifolds have been classified in \cite{WYZ}. It
is an interesting question if such a difference also exists for
noncompact $s$-Gauduchon flat Hermitian manifolds when $s$ is other
than $0$ and $2$.

\vsv

The paper is organized as follows. In \S 2, we collect some known
results and fix the notations. In \S 3, we give proofs to Theorem
\ref{intro2} through Proposition \ref{intro5}. In \S 4, we prove
Theorem \ref{intro1}.

\vsv

\vsv

\vsv

\vs

\section{Preliminaries}

\vsv

We begin with a  Hermitian manifold $(M^n,g)$. We will follow the
notations of \cite{YZ} for the most part. Denote by $\nabla$,
$\nabla^c$ the Riemannian (aka Levi-Civita) and the Chern
connection, respectively. Denote by $R$, $R^c$ the curvature tensors
of these two connections, and by $T^c$ the torsion tensor of
$\nabla^c$. Under a local unitary frame $\{ e_1, \ldots , e_n\} $ of
type $(1,0)$ tangent vectors, $T^c$ has components
$$ T^c(e_i, e_j) = \sum_{k=1}^n 2\ T_{ij}^k e_k,
\ \ \ \ \  \ T^c(e_i, \overline{e_j}) =0. \ \ \ \ \ $$

Note the coefficient $2$ above, which is unconventional but makes
some of the subsequent formula simpler. We will write $e= \ ^t(e_1,
\ldots , e_n)$ as a column vector. Write $\varphi = \
^t\!(\varphi_1, \ldots , \varphi_n)$ the column vector of local
$(1,0)$-forms that are dual to $e$. As in \cite{YZ}, let us write
$$ \nabla^c e = \theta e, \ \ \nabla e =
\theta_1 e + \overline{\theta_2} \overline{e}, $$ so $\theta$ and
$\Theta = d\theta - \theta \wedge \theta $ are the matrices of
connection and curvature of $\nabla^c$ under the unitary frame $e$,
while $\hat{\theta}$ and $\hat{\Theta } = d \hat{\theta} -
\hat{\theta} \wedge \hat{\theta} $ are the matrices of connection
and curvature of $\nabla$ under the frame $\{ e, \overline{e}\}$,
with
$$ \hat{\theta } = \left[ \begin{array}{ll} \theta_1 & \overline{\theta_2 } \\
\theta_2 & \overline{\theta_1 }  \end{array} \right] , \ \  \ \ \
\hat{\Theta } = \left[ \begin{array}{ll} \Theta_1 & \overline{\Theta}_2  \\
\Theta_2 & \overline{\Theta}_1   \end{array} \right]. $$

The structure equations are:
\begin{eqnarray}
d \varphi & = & - \ ^t\!\theta \wedge \varphi + \tau,  \label{formula 1}\\
d  \theta & = & \theta \wedge \theta + \Theta.
\end{eqnarray}
where $\tau$ is the column vector of the
torsion $2$-forms under the local frame $e$, and
\begin{eqnarray}
d\varphi & = & - \ ^t\! \theta_1 \wedge \varphi - \ ^t\! \theta_2
\wedge \overline{\varphi }\\
\Theta_1 & = & d\theta_1 -\theta_1 \wedge \theta_1 -\overline{\theta_2} \wedge \theta_2, \\
\Theta_2 & = & d\theta_2 - \theta_2 \wedge \theta_1 - \overline{\theta_1 } \wedge \theta_2,  \label{formula 7}.
\end{eqnarray}

The entries of $\tau$ are $(2,0)$ forms, and the entries of $\Theta$
are all $(1,1)$ forms. Taking exterior differentiation of the above
equations, we get the two Bianchi identities:
\begin{eqnarray}
d \tau & = & - \ ^t\!\theta \wedge \tau + \ ^t\!\Theta \wedge \varphi,
\label{formula 3} \\
d  \Theta & = & \theta \wedge \Theta - \Theta \wedge \theta.
\end{eqnarray}

From \cite{YZ}, we know that when $e$ is unitary, we have the following
simple formula
\begin{equation}
\tau_k = \sum_{i,j=1}^n T_{ij}^k \varphi_i\wedge \varphi_j,
\ \ (\theta_2)_{ij} = \sum_{k=1}^n \overline{T^k_{ij}} \varphi_k ,
\ \ \gamma'_{ij} = \sum_{k=1}^n T^j_{ik} \varphi_k
\end{equation}
where $\gamma'$ is the $(1,0)$ part of $\gamma = \theta_1-\theta =
\gamma ' - \gamma '^{\ast }$. We will also denote by $\eta =
\mbox{tr}(\gamma')$ the torsion $1$-form, aka Gauduchon $1$-form
(\cite{Gauduchon}), so we have
\begin{eqnarray}
 & & \eta =  \sum_{i=1}^n \eta_i \varphi_i =
 \sum_{i,k=1}^n T^k_{ki} \varphi_i \\
& & \partial \omega^{n-1} = -2 \ \eta \wedge \omega^{n-1}
\end{eqnarray}
where $\omega = \sqrt{-1} \ ^t\!\varphi \wedge \overline{\varphi}$
is the K\"ahler $(1,1)$-form. From the last equation above, we get
\begin{equation}
\partial \overline{\partial } \omega^{n-1} =
2\ (\overline{\partial } \eta + 2\eta \wedge \overline{\eta })
\wedge \omega^{n-1}
\end{equation}

By Lemma 2 of \cite{WYZ}, we know that the matrix of connection for
$\nabla^b$ under $e$ is given by $\theta + 2\gamma$, therefore we obtain
the following:

\begin{lemma}\label{lemma 2-1}
Given a Hermitian manifold $(M^n,g)$, the matrix of connection forms
for the $s$-Gauduchon connection under the frame $e$ is given by
\begin{equation}
\theta^s = \theta + s \gamma
\end{equation}
\end{lemma}

\vsv

Next let us recall the conformal change formula. Let $\tilde{g}=e^{2u}g$ be
 a metric  conformal to $g$, where $u$ is a smooth real valued function.
 Locally we can take $\tilde{\varphi} =e^u\varphi$ and $\tilde{e} = e^{-u}e$,
 so $\tilde{e}$ is a local unitary frame for $(M^n,\tilde{g})$, with
 $\tilde{\varphi}$ its dual coframe.

\vsv

Let $\tilde{\theta}$, $\widetilde{\Theta}$, and $\tilde{\tau}$ be
respectively the matrices or column vector of the Chern connection,
curvature, and torsion for the metric $\tilde{g}$ under the unitary
frame $\tilde{e}$. From \S 5 of \cite{YZ}, we have the following:
$$ \tilde{\theta} = \theta + (\partial u - \overline{\partial }u)I ,
\ \ \ \widetilde{\Theta} = \Theta - 2 \partial  \overline{\partial }u I, $$
and
$$ \tilde{\tau } = e^u (\tau + 2 \ \partial u \wedge \varphi ). $$
This leads to the following
\begin{equation}
\widetilde{T^k_{ij}} = e^{-u}
\big[  T^k_{ij} + u_i \delta_{jk} - u_j \delta_{ik}  \big] ,
\end{equation}
where $u_i=e_i(u)$.

\vsv

Next, recall that $\eta_i = \sum_k T^k_{ki}$, and let us denote by
$$ |\eta|^2 = \sum_{i=1}^n |\eta_i|^2,  \ \ \  |T|^2 =
\sum_{i,j,k=1}^n |T^i_{jk}|^2 .$$
Both are independent of the choice of unitary frames thus are
well defined global functions on the manifold $M^n$.

\vsv

If $\tilde{g}$ is K\"ahler, namely, if $\tilde{\tau}=0$, then we have
$$ T^k_{ij} = u_j\delta_{ik} - u_i\delta_{jk} $$
for any indices $i$, $j$, $k$. This leads to the following:

\begin{lemma} \label{lemma 2-2}
If $(M^n,g)$ is a Hermitian manifold that is locally conformally
K\"ahler, then it holds $$|T|^2 = \frac{2}{n-1}|\eta |^2.$$
\end{lemma}

\begin{proof}
From the identity right above the statement of the lemma, we
know that $T^k_{ij}=0$ when $k\notin \{ i,j\}$, and
$T^i_{ij}=u_j$ if $i\neq j$. So $\eta_j = (n-1)u_j$. From this,
we get $|\eta|^2 = (n-1)^2 |du|^2$, and $|T|^2 = 2(n-1)|du|^2$,
where $|du|^2= |u_1|^2 + \cdots + |u_n|^2$. Thus
$|T|^2 = \frac{2}{n-1}|\eta |^2$, and the lemma is proved.
\end{proof}

\vsv

\vsv

\vs

\section{The $s$-Gauduchon flat manifolds}

\vsv

Throughout this section, we will assume that $(M^n,g)$ is a
Hermitian manifold with flat $s$-Gauduchon connection $\nabla^s$,
where $s\neq 0$. For any $p\in M$, there always exists a unitary
frame $e$ in a neighborhood of $p$ such that is $e$ is
$\nabla^s$-parallel. That is, $\theta^s=0$. Note that such a frame
is unique up to changes by constant valued unitary matrices. Let us
fix such a local frame $e$. Specializing the structure equations and
Bianchi identities with our condition $\theta = -s \gamma $, and
using the fact that  $^t\!\gamma' \varphi = - \tau$, we get
\begin{eqnarray}
\partial \varphi & = & (s-1) \ ^t\! \gamma' \varphi \\
\overline{\partial} \varphi & = & -s \ \overline{\gamma'} \varphi \\
\partial \gamma' & = & -s \ \gamma' \gamma' \\
^t\!\Theta & = & -s \ (\overline{\partial} \  ^t\! \gamma' - \partial \overline{\gamma' }) - s^2 (\overline{\gamma' } \ ^t\! \gamma' + \ ^t\! \gamma'  \overline{\gamma' } )
\end{eqnarray}
where the third equation is because of the vanishing of the $(2,0)$-component of $\Theta$ and the fact that $s\neq 0$. The $(2,1)$-part of the first Bianchi identity $d\tau = \ ^t\!\Theta \varphi - \ ^t\!\theta \tau$ leads to
\begin{equation}
[ \ (s-1) \ \overline{\partial} \  ^t\! \gamma' - s \ \partial \overline{\gamma'} + s(s-1) \ ( \  ^t\! \gamma' \overline{\gamma'}  + \overline{\gamma'}  \  ^t\! \gamma'  )\  ] \ \varphi = 0.
\end{equation}

\vsv

\begin{lemma} \label{lemma 3-1}
Suppose a Hermitian manifold $(M^n,g)$ is $s$-Gauduchon flat, where
$s\neq 0$. Let $e$ be a local $\nabla^s$-parallel unitary frame,
then the torsion components satisfy
\begin{eqnarray}
& & T^{\ell }_{ij,k}  -  T^{\ell}_{ik,j}  =  \sum_{r=1}^n \{ 2 (1-s) T^{\ell}_{ir} T^r_{jk} +  s T^{\ell}_{jr} T^r_{ik} - s T^{\ell}_{kr} T^r_{ij} \}  \\
& & (n-2)(s-1) \sum_{r=1}^n  \{ T^{\ell}_{ir} T^r_{jk}  + T^{\ell}_{jr} T^r_{ki} + T^{\ell}_{kr} T^r_{ij}  \} = 0 \\
& & 2(s-1) T^k_{ij,\overline{\ell } } + s (  \overline{  T^i_{k\ell , \overline{j} } } - \overline{   T^j_{k\ell , \overline{i}}     } ) = \sum_{r=1}^n \{ 2(s-s^2) T^r_{ij} \overline{ T^r_{k\ell }  } + 2(s-s^2) ( T^k_{ri} \overline{ T^j_{r\ell } } -  T^k_{rj} \overline{ T^i_{r\ell } }  ) + \\
\nonumber & & \hspace{5.8cm} + \  s^2 ( T^{\ell}_{ri} \overline{ T^j_{rk } } -  T^{\ell }_{rj} \overline{ T^i_{rk } }   )  \}
\end{eqnarray}
for any $1\leq i,j,k,\ell \leq n$, where the index after comma denotes the covariant derivative with respect to $\nabla^s$.
\end{lemma}

\begin{proof}
The first identity comes from the fact that $\partial \gamma' = -s \gamma' \gamma' $ since $s\neq 0$. When $n\geq 3$, if we combine the identity $\partial \gamma' = - s \gamma ' \gamma'$ with the $(3,0)$-part of the first Bianchi identity, $\partial \tau = s \ ^t\gamma' \tau $, we get $(s-1) \ ^t\! \gamma' \ ^t\! \gamma' \varphi =0$. This leads to the second equation. The third equation comes from $(18)$.
\end{proof}

\vsv

By the first two equations in the above lemma, we get

\begin{lemma}\label{lemma 3-2}
Let $(M^n,g)$ be a Hermitian manifold that is $s$-Gauduchon flat,
where $s\neq 1$ and $n\geq 3$. Then for any indices $i,j,k,\ell$, it
holds
\begin{equation}
T^{\ell }_{ij,k}  -  T^{\ell}_{ik,j}  =  (2-s) \sum_{r=1}^n  T^{\ell}_{ir} T^r_{jk}
\end{equation}
\end{lemma}

This identity shows that the case of the Bismut connection $(s=2)$
is special as the right hand side would vanish when $s=2$. Next,
recall that $\eta_i = \sum_k T^k_{ki}$, and let us denote by
$$ |\eta|^2 = \sum_{i=1}^n |\eta_i|^2,  \ \ \
|T|^2 = \sum_{i,j,k=1}^n |T^i_{jk}|^2, \ \ \ \mbox{and} \ \ \chi =
\sum_{i=1}^n \eta_{i, \overline{i}}. $$ In the last equation of
Lemma \ref{lemma 3-1}, choose $i=k$ and $j=\ell $ and sum them up
from $1$ to $n$, we get
$$ 2(s-1)\chi + 2s \overline{\chi } = s^2 |T|^2 + 2(s-\frac{3}{2}s^2) |\eta|^2. $$
Since the right hand side is a real number, we see that $\chi$ must be real, so we get

\begin{lemma} \label{lemma 3-3}
Under a local unitary $\nabla^s$-parallel frame $e$, the quantity
$\chi = \sum_{i=1}^n \eta_{i, \overline{i}}$ satisfies the identity
\begin{equation}
(2s-1)\chi = \frac{s^2}{2} |T|^2 + s(1-\frac{3}{2}s) |\eta |^2.
\end{equation}
\end{lemma}

\vsv

Now we are ready to prove Proposition \ref{intro5}. For
$s=\frac{1}{2}$, the above identity gives $\frac{1}{8} |\eta |^2 +
\frac{1}{8} |T|^2 = 0$, so $T=0$ everywhere, and $g$ is K\"ahler. If
$s\neq 0$ and $\eta =0$, then $\chi =0$, so the above equation leads
to $T=0$ again. This proves Proposition \ref{intro5}.

\vsv

Next, by the equation $\overline{\partial }\varphi = -s\overline{\gamma'}\varphi$, we get
$$ \overline{\partial } \eta = - \sum_{i,j=1}^n \big(\eta_{i,\overline{j} } + s \sum_{k=1}^n \eta_k \overline{T^i_{jk} } \big) \ \varphi_i \wedge \overline{\varphi_j}. $$
Since  \ $ n \sqrt{-1} \ \big(\sum_{i,j=1}^n a_{ij} \varphi_i \wedge \overline{\varphi_j} \big) \wedge \omega^{n-1} = \big( \sum_{i=1}^n a_{ii} \big) \ \omega^n$, we obtain the following
\begin{equation}
n \sqrt{-1} \ \overline{\partial } \eta \wedge \omega^{n-1}  = - (\chi + s| \eta |^2 ) \ \omega^n .
\end{equation}
Combining this with Lemma \ref{lemma 3-3}, we get

\begin{lemma} \label{lemma 3-4}
Let $(M^n,g)$ be a $\nabla^s$-flat Hermitian manifold. Then it holds
\begin{equation}
(2s-1)n\sqrt{-1} \partial \overline{\partial }  \omega^{n-1}
= \big[ (8s-s^2-4)\ |\eta|^2 - s^2 |T|^2 \big] \ \omega^n
\end{equation}
\end{lemma}

\vsv

When the $s$-Gauduchon flat manifold $M^n$ is compact, the
integral of the left hand side is zero. Thus we immediately get
the following
\begin{lemma}\label{lemma 3-5}
Let $(M^n,g)$ be a compact $\nabla^s$-flat Hermitian manifold. Then
it holds
\begin{equation}
 (8s-s^2-4) \int_M |\eta |^2 \  \omega^n = s^2 \int_M |T|^2 \ \omega^n
\end{equation}
\end{lemma}

\vsv

Since the two roots of $8s -s^2-4$ are $4\pm 2\sqrt{3}$, when either
$s\geq 4+ 2\sqrt{3}$, or $s\leq 4-2\sqrt{3}$ and $s\neq 0$, the left
hand side is nonpositive, which will force $T$ to be identically
zero, meaning that $g$ is K\"ahler. So we have proved Theorem
\ref{intro2}.

\vsv

Theorem \ref{intro2} means that when $M^n$ is compact, we may assume
that $4-2\sqrt{3} < s < 4+2\sqrt{3}$, or approximately,
$$0.54<s<7.46$$

Next, note that $\eta_i = \sum_k T^k_{ki}$ is the sum of at most
$(n-1)$ terms, since $T^i_{ii}=0$. By the inequality $|a_1 + \cdots
+ a_{n-1}|^2 \leq (n-1) ( |a_1|^2 + \cdots + |a_{n-1}|^2)$, we know
that $$ |\eta |^2 \leq (n-1) |T|^2. $$ Plug this into the identity
in Lemma \ref{lemma 3-5}, we get
\begin{equation}
F(s)\int_M|T|^2 \omega^n \leq 0,  \ \ \ \ \mbox{where} \ \ \  \
F(s) = ns^2 -8(n-1)s+4(n-1).
\end{equation}

Clearly, when $F(s)>0$, the above inequality implies $T=0$ everywhere.
The two roots of $F$ are exactly the two dimension-dependent constants
$$ a_n^{\pm} : = \ \frac{1}{n} \big[ 4(n-1) \pm 2\sqrt{3n^2-7n+4 } \big] \  $$ given in the introduction. Thus we have completed the proof of
Theorem \ref{intro3}.

\vsv

Note that the improvement of Theorem \ref{intro3} to Theorem
\ref{intro2} works better when $n$ is smaller, as $a^{\pm}_n
\rightarrow 4\pm 2\sqrt{3}$ when $n\rightarrow \infty$. When $n=3$,
for instance, $a^{\pm}_3 = \frac{2}{3}(4\pm \sqrt{10})$, so we just
need to consider the range
$$ 0.56 \leq s \leq 4.77 $$
for any potential non-K\"ahler $s$-Gauduchon flat compact threefold.

\vsv

Next let us consider the locally conformally K\"ahler case. In this
case we have the relationship $|\eta|^2 = \frac{(n-1)}{2}|T|^2$ by
Lemma \ref{lemma 2-2}. Plug this again into the identity in Lemma
\ref{lemma 3-5}, we get
\begin{equation}
\big[ 8(n-1)s - 4(n-1) - (n+1)s^2 \big] \int_M |T|^2 \omega^n = 0
\end{equation}
for compact Hermitian manifold $(M^n,g)$ that is $\nabla^s$-flat and
locally conformally K\"ahler. When the coefficient is not zero, it
will force $T=0$, so $g$ would be K\"ahler. The zeroes of the
polynomial are precisely the two constant $b_n^{\pm}$ given in the
introduction. So we have completed the proof of Theorem \ref{intro4}
here.

\vsv

Now let us restrict ourselves to the two dimensional case. In this case,
$\eta_1=-T^2_{12}$, $\eta_2=T^1_{12}$, so $|T|^2=2 |\eta |^2$. Plug this
 into the identity in Lemma \ref{lemma 3-5}, we get
\begin{equation}
(3s-2)(s-2) \int_{M^2} |T|^2 \omega^2 = 0.
\end{equation}
Therefore we obtained the following

\begin{lemma} \label{lemma 3-6}
Let $(M^2,g)$ be a compact Hermitian surface with flat $s$-Gauduchon
connection. If $s \neq 2$  and $s\neq \frac{2}{3}$, then $g$ is
K\"ahler.
\end{lemma}

Therefore in order to prove Theorem \ref{intro1}, we just need to
deal with the $s=\frac{2}{3}$ case, namely, the minimal Gauduchon
connection $\nabla^{\frac{2}{3}}$.

\vs

\vs

\section{Surfaces with  flat minimal connection}\label{surface}

\vsv

In this section, we will assume that $(M^2,g)$ is a compact
Hermitian surface with flat $\nabla^{\frac{2}{3}}$, and our goal is
to use the Bochner formula to conclude that $g$ must be K\"ahler,
thus proving Theorem \ref{intro1} stated in the introduction.

\vsv

Let $s=\frac{2}{3}$, and $e$ be a $\nabla^s$-parallel local unitary
frame. Let us denote by $a=T^1_{12}= \eta_2$, $b=T^2_{12}=-\eta_1$.
Then $\lambda = |a|^2+|b|^2=|\eta|^2$ is a nonnegative smooth
function on $M^2$. Since $2(1-s)=s$, in the first identity of Lemma
\ref{lemma 3-1}, if we let $i=k=1$, $j=2$, we get
$$ T^{\ell }_{12,1} = -2s T^{\ell }_{12}T^2_{12} .$$
Similarly, if we let $i=k=2$ and $j=1$, we get
$$ T^{\ell }_{12,2} = 2s T^{\ell }_{12}T^1_{12} .$$
That is, we have the following
\begin{equation}
 a_1 = -2s \ ab, \ \ \  b_1 = -2s \ b^2, \ \ \ a_2 = 2s \ a^2, \
  \ \  b_2=2s \ ab
\end{equation}
Next, we look at the last identity in Lemma \ref{lemma 3-1}. Let
$i=1$, $j=2$, and $k=\ell$, we get
$$ T^k_{12,\overline{k}} = 2s T^k_{12} (\overline{T^2_{2k} } +
\overline{ T^1_{1k} } ), $$
that is,
\begin{equation}
 a_{\overline{1}} = -2s \ a\overline{b}, \ \ \  b_{\overline{2}}
 = 2s \ \overline{a}b.
\end{equation}
Similarly, by letting respectively $i=k=1$, $j=\ell =2$ or $i=\ell
=1$, $j=k=2$ in the last identity of Lemma \ref{lemma 3-1}, and
using the results for $a_{\overline{1}}$ and $b_{\overline{2}}$
above, we get
\begin{equation}
 a_{\overline{2}} = 2s \ |a|^2 , \ \ \  b_{\overline{1}} = -2s \ |b|^2.
\end{equation}
From this, we get
\begin{eqnarray*}
\lambda_1 & = & -4s \lambda b , \ \ \ \ \
\lambda_2 \ = \ 4s \lambda a \\
 \lambda_{1\overline{1}} & = & 24s^2 \lambda |b|^2 ,   \ \
  \lambda_{2\overline{2}}\  = \ 24s^2 \lambda |a|^2
\end{eqnarray*}
Thus
$$ \sum_i \lambda_{i\overline{i}} = 24s^2\lambda^2, \ \ \ \
\sum_i \lambda_i \overline{\eta_i} = 4s\lambda^2, \ \ \ \
\sum_i |\lambda_i|^2 = 16s^2 \lambda^3. $$

Now let us consider the smooth function $f=\log (\lambda + \eps )$
on $M^2$, where $\eps >0$ is a constant. Since
$\overline{\partial}\varphi = -s \overline{\gamma'} \varphi$, we have
\begin{eqnarray*}
\partial \overline{\partial} f & = & - \overline{\partial} \big( \sum_i f_i \varphi_i\big)  \\
& = & \sum_{i,j} \big( f_{i\overline{j}} + s \sum_k f_k \overline{T^i_{jk}} \big) \ \varphi_i \wedge \overline{ \varphi_j}
\end{eqnarray*}
Therefore,
\begin{eqnarray*}
2\sqrt{-1} \partial \overline{\partial} f \wedge \omega & = & \big( \sum_i f_{i\overline{i}} + s\sum_i f_i \overline{\eta_i} \big)\ \omega^2 \\
& = & \big( \frac{1}{\lambda +\eps } \sum_i \lambda_{i\overline{i}} - \frac{1}{(\lambda +\eps )^2 } \sum_i |\lambda_i|^2   + \frac{s}{\lambda +\eps }  \sum_i \lambda_i \overline{\eta_i} \big)\ \omega^2 \\
& = & 4s^2 \frac{(3\lambda + 7\eps )} {(\lambda + \eps )^2}
\lambda^2 \omega^2
\end{eqnarray*}

On the other hand, by Lemma \ref{lemma 3-4}, specialized to our case
of $n=2$ and $s=\frac{2}{3}$, we get $\partial \overline{\partial}
\omega =0$. So if we integrate the above equality, we get $\lambda
\equiv 0$ on $M^2$, that is, $(M^2,g)$ is K\"ahler. This completes
the proof of Theorem \ref{intro1}.

\vs


\vs

\end{document}